\documentclass[12pt]{article}
\usepackage{amsmath,amssymb,graphicx,amsthm, fullpage}

\newtheorem{theorem}{Theorem}
\newtheorem*{theorem:repeat}{\tref{butterflystab}}
\newtheorem*{theorem:repeatmain}{\tref{main}}
\newtheorem{lemma}[theorem]{Lemma}
\newtheorem{corollary}[theorem]{Corollary}

\newtheorem{prop}[theorem]{Proposition}

\newtheorem{claim}[theorem]{Claim}

\newtheorem{construction}[theorem]{Construction}
\newtheorem{observation}[theorem]{Observation}

\newcommand\lref[1]{Lemma~\ref{lem:#1}}
\newcommand\tref[1]{Theorem~\ref{thm:#1}}
\newcommand\cref[1]{Corollary~\ref{cor:#1}}
\newcommand\csref[1]{Construction~\ref{cons:#1}}
\newcommand\clref[1]{Claim~\ref{clm:#1}}

\newcommand\pref[1]{Proposition~\ref{prop:#1}}

\newcommand\oref[1]{Observation~\ref{obs:#1}}

\newcommand\cC{{\mathcal C}}

\title{On the number of cycles in a graph with restricted cycle lengths}

\author{D\'aniel Gerbner\thanks{Hungarian Academy of Sciences, Alfr\'ed R\'enyi Institute of Mathematics, P.O.B. 127, Budapest H-1364, Hungary.}, Bal\'azs Keszegh\thanks{Hungarian Academy of Sciences, Alfr\'ed R\'enyi Institute of Mathematics, P.O.B. 127, Budapest H-1364, Hungary. Research supported by the J\'anos Bolyai Research Fellowship of the Hungarian Academy of Sciences and the National Research, Development and Innovation Office -- NKFIH under the grant PD 108406 and the grant K 116769.}, Cory Palmer\thanks{Department of Mathematical Sciences,
University of Montana, Missoula, Montana 59812, USA. Research supported by University of Montana UGP grant M25364.}, Bal\'azs Patk\'os\thanks{Hungarian Academy of Sciences, Alfr\'ed R\'enyi Institute of Mathematics, P.O.B. 127, Budapest H-1364, Hungary.}}

\begin{document}

\maketitle

\begin{abstract}
Let $L$ be a set of positive integers. We call a (directed) graph $G$ an $L$\emph{-cycle graph} if all cycle lengths in $G$ belong to $L$. Let $c(L,n)$  be the maximum number of cycles possible in an $n$-vertex $L$-cycle graph (we use $\vec{c}(L,n)$ for the number of cycles in directed graphs). In the undirected case we show that for any fixed set $L$, we have $c(L,n)=\Theta_L(n^{\lfloor k/\ell \rfloor})$  where $k$ is the largest element of $L$ and $2\ell$ is the smallest even element of $L$ (if $L$ contains only odd elements, then $c(L,n)=\Theta_L(n)$ holds.) We also give a characterization of $L$-cycle graphs when $L$ is a single element.

In the directed case we prove that for any fixed set $L$ we have $\vec{c}(L,n)=(1+o(1))(\frac{n-1}{k-1})^{k-1}$, where $k$ is the largest element of $L$. We determine the exact value of $\vec{c}(\{k\},n)$ for every $k$ and characterize all graphs attaining this maximum.
\end{abstract}

\section{Introduction}

In this paper we examine graphs that contain only cycles of some prescribed lengths (where the \emph{length} of a cycle or a path is the number of its edges). Let $L$ be a set of positive integers. We call a graph $G$ an $L$-\emph{cycle graph} if all cycle lengths in $G$ belong to $L$. That is, $L$ can be thought of as the list of ``allowed'' cycle lengths. We restrict our attention to graphs with no loops or multiple edges, so when $G$ is undirected $L$ contains only integers greater than or equal to $3$; when $G$ is directed $L$ contains only integers greater than or equal to $2$. 

This problem has two main motivations. First is the classical Tur\'an number for cycles, i.e., the question of determining the maximum possible number of edges in a graph with no cycles of certain specified lengths. A standard result in this context is the even cycle theorem of Bondy and Simonovits \cite{BoSi} that states that an $n$-vertex graph with no cycle of length $2k$ has at most $cn^{1+1/k}$ many edges. 

Instead of forbidding one cycle length, we seek to forbid most cycle lengths and instead focus on a set $L$ of permitted cycle lengths. Generally the size of this list will be bounded when compared to the number of vertices $n$.

The second motivation is the study of the number of substructures in a fixed class of graphs. In particular, for fixed graphs $H$ and $F$ counting the number of subgraphs $H$ in a graph that contains no $F$ subgraph. For a general overview for graphs, see Alon and Shikhelman \cite{AlSh}. Two representative examples are as follows. Erd\H os \cite{Er} conjectured that the maximum number of cycles of length $5$ in an $n$-vertex triangle-free graph is $(n/5)^5$. Gy\H ori \cite{Gy} proved that this maximum is at most $1.03 \cdot (n/5)^5$. Later, Grzesik \cite{Gr} and independently Hatami, Hladk\'y, Kr\'al', Norine and Razborov \cite{HaETAL} proved the conjecture of Erd\H os.

 Bollob\'as and Gy\H ori \cite{BoGy} posed a similar question: determine the maximum possible number of triangles in an $n$-vertex graph with no cycle of length $5$. They proved an upper bound of $(5/4)n^{3/2} +o(n^{3/2})$ which gives the correct order of magnitude in $n$.

It is not hard to show that an $n$-vertex graph with all cycles of lengths in $L$ has at most $|L| \cdot n$ many edges (see \pref{edges}). A more interesting question is to determine the maximum number of cycles possible in an $n$-vertex $L$-cycle graph $G$. We denote this maximum by $c(L,n)$. Observe that if all cycles lengths in $L$ are larger than $n$ then $c(L,n)=0$. In particular $c(\{k\},n)=0$ whenever $k>n$.

Before we can state our results, we introduce some definitions.
The \emph{distance} between two vertices is the length of the shortest path between them.
The \textit{Theta-graph} $T_{r,\ell}$ is the graph that consists of $r\ge 1$ vertex-independent paths of length $\ell$ joining two vertices $u$ and $v$. Note that $T_{r.\ell}$ contains $\binom{r}{2}$ cycles of length $2\ell$ and no other cycles. The vertices $u,v$ are called the \textit{main vertices} of $T_{r,\ell}$. The $r$ vertex-independent paths between $u$ and $v$ are called the \emph{main paths}. Two vertices at distance $\ell$ are called \emph{opposite vertices}. Note that any pair of opposite vertices are at distance $a$ from a main vertex where $0\le a<\ell$.

Fix $k$ and $\ell$ such that $k > 2\ell$.
Consider a cycle of length $k$ with vertex set $\{1,2,3,\dots, k\}$. Suppose $i_1 < i_2 < \dots < i_s$ are vertices on the cycle each of distance at least $\ell$ from each other. Now, for each $i_j$ let us add an arbitrary number of paths of length $\ell$ from $i_j$ to $i_j+\ell$ by introducing $\ell-2$ many new vertices for each path. Observe that the collection of paths from $i_j$ to $i_j+\ell$ form a $T_{r,\ell}$. We denote the class of graphs that can be formed in this way by $\mathcal{C}(k,\ell)$. It is easy to see that these graphs have cycles of lengths $k$ and $2 \ell$ only.

%

A \emph{block} of a graph $G$ is a maximal connected subgraph without a cutvertex. In particular, every block in a graph is either a maximal $2$-connected subgraph, a cut-edge, i.e., an edge whose removal disconnects the graph, or an isolated vertex. In the case when $L$ is a set of odd integers or $|L|\leq 2$ and $L$ contains at most one even integer, we will characterize $L$-cycle graphs by describing their $2$-connected blocks.

\begin{theorem}\label{2-conn}
\label{thm:exactodd} 
\textbf{(a)} Let $L$ be a set of odd integers, each at least $3$. If $G$ is an $L$-cycle graph, then each $2$-connected block of $G$ is a cycle with length in $L$.

\vskip 0.2truecm

\textbf{(b)} Let $G$ be a $\{2k\}$-cycle graph on $n$ vertices.  Then every 2-connected block of $G$ is a $T_{r,k}$ for some $r\ge 2$.

\vskip 0.2truecm

\textbf{(c)} Let $G$ be a $\{2k+1,2\ell\}$-cycle graph on $n$ vertices.  Then every 2-connected block of $G$ is a  $T_{r,\ell}$ for some $r\ge 2$ or a graph in $\mathcal{C}(2k+1,2\ell)$.

\end{theorem}




When all the cycles are odd, or only one fixed cycle length is allowed, it is not hard to determine the maximum number of cycles using this structural result.

\begin{corollary} \label{cor:undir}
\textbf{(a)} If $L$ is a set of odd integers with smallest element $2k+1>1$. Then 
\[c(L,n) = \left\lfloor \frac{n-1}{2k} \right\rfloor.\]

In particular, if $L$ is a single odd integer $2k+1>1$ then, \[c(\{2k+1\},n) = \left\lfloor \frac{n-1}{2k} \right\rfloor.\]

\textbf{(b)} If $L$ is a single even integer $2k$, then

\[c({\{2k\},n}) =  \binom{\lfloor \frac{n-2}{2k-2} \rfloor}{2}.\]

\end{corollary}


In the cases not covered by Corollary \ref{cor:undir} we can give the order of magnitude of $c(L,n)$.

\begin{theorem}
\label{thm:ordermagn}
Let $L$ be a set of integers with smallest even element $2\ell$ and largest element $k$. Then
 \[c(L,n)=\Theta_L(n^{\lfloor k/\ell \rfloor}).\]
\end{theorem}

Let $\vec{c}(L,n)$ denote the maximum number of directed cycles that an $n$-vertex directed graph $G$ can contain, provided the length of every directed cycle in $G$ belongs to $L$. Again, trivially $\vec{c}(L,n)=0$ (and thus $\vec{c}(\{k\},n)=0$) if every cycle length in $L$ is larger than $n$.

\begin{theorem}
\label{thm:dirasy} Let $L$ be a set of integers with largest element $k$. Then
\[\vec{c}(L,n)=\left(\frac{n-1}{k-1}\right)^{k-1}+O(n^{k-2}).\]

\end{theorem}

When $L$ is a single cycle length, then we can determine $\vec{c}(L,n)$ exactly.

\begin{theorem}
\label{thm:direxact}
For every $2\leq k \le n$ we have
\[
\vec{c}(\{k\},n)=\prod_{i=0}^{k-2}\left\lfloor\frac{n-1+i}{k-1}\right\rfloor.
\]
\end{theorem}

In fact, we can characterize the graphs attaining this maximum. We postpone their description until Section~\ref{directed-section}.


\section{Undirected graphs}

It is well-known (see e.g.\ \cite{D-book}) that every 2-connected graph has an \emph{ear-decomposition}, i.e. it can be constructed by starting with a cycle and in each step adding a path (called an \emph{ear}) between two distinct vertices of the graph from the previous step (where the internal vertices of the path are new vertices).

\begin{observation}\label{obs:theta} Let $u$ and $v$ be vertices in $T_{r,\ell}$ with $r\ge 3$ but not the main vertices, then there is a path of length greater than $\ell$ between $u$ and $v$. If $u$ and $v$ are opposite vertices at distance $0<a<k$ from a main vertex, then there is a path of length $\ell+2a$ between them.

\end{observation}

\begin{proof}[Proof of \tref{exactodd}.]

In order to prove the theorem we may restrict our attention to a $2$-connected block $H$. In all cases we proceed by induction on the number of ears in an ear-decomposition. A $2$-connected block with $0$ ears is a cycle which satisfies the base case of all three parts of the theorem (as a cycle of length $2\ell$ is a $T_{2,\ell}$).

First we prove \textbf{(a)}. Suppose the ear-decomposition of $H$ includes at least one ear. Consider the first path added to a cycle in the ear-decomposition of $H$. Suppose the ear has end-vertices $u$ and $v$. Among the three paths between $u$ and $v$, two have the same parity. These two paths of the same parity form an even cycle; a contradiction. Thus $H$ is a cycle.

\smallskip

Next we prove \textbf{(b)}. By induction, after removing an ear $P$ from $H$ we are left with the graph $T_{r,\ell}$. Suppose $P$ is of length $i$ and has end-vertices $u$ and $v$ in $T_{r,\ell}$. There are disjoint paths of length $j$ and $2\ell-j$
between $u$ and $v$ in $T_{r,\ell}$. Clearly this is only possible if $i=j=\ell$. If $r=2$ then after adding $P$ we get a $T_{3,\ell}$. If $r>2$, then if $u$ and $v$ are the main vertices of $T_{r,\ell}$ then we get a $T_{r+1,\ell}$. Otherwise, by \oref{theta} there is a path of length greater than $\ell$ between $u$ and $v$ in $T_{r,\ell}$. Together with $P$ this creates a cycle of length greater than $2\ell$, a contradiction.
\smallskip

Finally, we prove \textbf{(c)} using some elements of the proof of \textbf{(b)}. We distinguish two cases based on the graph resulting from the removal of an ear $P$ from $H$. Let $u$ and $v$ be the endpoints of $P$ and suppose $P$ is of length $i$.

\textit{Case 1}. Suppose after removing the ear $P$ from $H$ we are left with $T_{r,\ell}$, $r\ge 2$.  We have three disjoint paths in $H$ between $u$ and $v$ of lengths $i$, $j$ and $2\ell-j$. As $|L|=2$, two of these numbers must be equal. 

If $i=j$ (or $i=2\ell-j$) then there is a cycle of even length $i+j=2i=2\ell$ (or $i+2\ell-j=2i=2\ell$), thus $i=j=\ell$. If $r=2$  then after adding $P$ we get a $T_{3,\ell}$. If $r>2$, then $u$ and $v$ must be opposite vertices. If $u$ and $v$ are the main vertices of $T_{r,\ell}$ then we get a $T_{r+1,\ell}$. Otherwise, by \oref{theta} there is a path of length $\ell+2a$ ($a\geq 1$) between $u$ and $v$ in $T_{r,\ell}$. This path, together with $P$ creates an even cycle of length $2\ell+2a\ne 2\ell$; a contradiction.

If $i\ne j$ and $i\neq 2\ell-j$, then $j=2\ell-j=\ell$ and $P$ is of length $i=2k+1-\ell$ (otherwise we have a forbidden cycle length). If $u$ and $v$ are the main vertices of $T_{r,\ell}$ (in particular if $r=2$) then after adding $P$ we get a graph as required. Otherwise, $r>2$ and $u$,$v$ are opposite vertices in $T_{r,\ell}$ that are not the main vertices. By \oref{theta} there is a path of length $\ell+2a$ ($a\geq 1$) between $u$ and $v$ in $T_{r,\ell}$. This path together with $P$ (that has length $2k+1-\ell$) creates an odd cycle of length $2k+1+2a>2k+1$; a contradiction.

\textit{Case 2}. Suppose after removing an ear $P$ from $H$ we are left with a graph in $\mathcal{C}(2k+1,\ell)$.
 Let $u$ and $v$ be the endpoints of $P$ and suppose $P$ is of length $i$.

\textit{Case 2.1}. Suppose that $u$ and $v$ belong to the same copy\footnote{For simplicity we refer to copies of $T_{r,\ell}$ even though in these copies the $r$'s may be different.}  of $T_{r,\ell}$ in $H$. In this case we proceed as in Case 1. If $r>2$ then we get a graph as required or we reach contradiction. However, if $r=2$, there is a difference from Case 1 as now the main vertices are uniquely defined. 

First, if $r=2$ and $i=j$ or $i=2\ell-j$ then again $i=j=\ell$ and $u$,$v$ are opposite vertices of the $T_{2,\ell}$. If they are the main vertices then the graph is as required. Otherwise $u,v$ are not the main vertices, they are at distance $a\geq 1$ from a main vertex and there is a cycle of odd length $2k+1-\ell+\ell+2a=2k+1+2a\ne 2k+1$ through them, a contradiction.

Second, if $r=2$ and $i\ne j$ and $i\ne 2\ell-j$ then, as in Case 1, $j=2\ell-j=\ell$ and so $u,v$ are opposite vertices. If they are not the main vertices of $T_{r,\ell}$ then we reach contradiction as in Case 1. On the other hand, if $u,v$ are the main vertices, then there are cycles of length $\ell+i$ and $2k+1-\ell+i$ through them. As $i\ne j=\ell$, the first cycle cannot have length $2\ell$ and so it has length $2k+1$, thus $i=2k+1-\ell$. Thus the second cycle has even length $2(2k+1-\ell)$ which must be equal to $2\ell$. Thus $2k+1-\ell=\ell$ and in turn $2k+1=2\ell$, a contradiction.

\textit{Case 2.2}.
Finally, consider the case when $u$ and $v$ are not in the same copy of a $T_{r,\ell}$. In this case they are on a cycle $C$ of length ${2k+1}$. There are already two disjoint paths $P_1$ of length $j$ and $P_2$ of length $2k+1-j$ between $u$ and $v$ in $C$ for some $j$. One of them creates an odd cycle, the other creates an even cycle with $P$. Without loss of generality we have $i+j=2\ell$ and $i+2k+1-j=2k+1$, thus $i=j=\ell$. 

If both $u$ and $v$ are main vertices of Theta-graphs then after adding $P$ we get a graph as required. Otherwise we can suppose that $u$ is in a Theta-graph $T$ but not a main vertex. In this case we build a cycle: we start in $u$, follow $P$, then go back to the main vertex $w_1$ (the one also in $P_1$) of the $T$, and go to the other main vertex $w_2$ on a path in $T$ avoiding $u$, then to $u$ on the shortest path in $T$ from $w_2$ to $u$. This gives an even cycle of length $2\ell+2a>2\ell$ (where $a$ is the distance of $u$ and $w_2$); a contradiction.
\end{proof}

\begin{proof}[Proof of \tref{ordermagn}.]
For the lower bound we construct the following graph. 

\begin{construction}\label{undirected-construction}
Take $\lfloor k/\ell \rfloor$ copies of the Theta-graph $T_{r,\ell}$ where $r=\lfloor {n}/{k} \rfloor$. Let $x_i,y_i$ for $i=1,2,\dots, \lfloor k/\ell \rfloor$ be the main vertices of these Theta-graphs.

For $i=1,2,\dots, \lfloor k/\ell \rfloor-1$, we identify $y_i$ with $x_{i+1}$. If $\ell$ divides $k$, then we also identify $y_{\lfloor k/\ell \rfloor}$ with $x_1$. Otherwise, we add a path of length $k-\lfloor k/\ell \rfloor \ell$ between $y_{\lfloor k/\ell \rfloor}$ and $x_1$.	
\end{construction}

The resulting graph in Construction~\ref{undirected-construction} is in the class $\mathcal{C}(k,2\ell)$ and therefore only has cycles of length $k$ and $2\ell$.
Furthermore,  we have used $k+\lfloor k/\ell\rfloor(r-1)(\ell-1) \le k+k(n/k-1)\leq n$ many vertices. In order to construct a graph on $n$ vertices we simply add isolated vertices as needed.
	
The number of cycles of length $2\ell$ is 
\[\left \lfloor \frac{k}{\ell} \right \rfloor \binom{r}{2} = \left\lfloor \frac{k}{\ell} \right \rfloor \binom{\left\lfloor \frac{n}{k} \right\rfloor}{2} = \Omega_L(n^2).\]

The number of cycles of length $k$ is
\[r^{\lfloor {k}/{\ell}\rfloor} = \left\lfloor \frac{n}{k} \right\rfloor^{\lfloor {k}/{\ell}\rfloor} = \Omega_L(n^{\lfloor {k}/{\ell}\rfloor}).\]

\begin{figure}[h]
	\begin{center}
		\includegraphics[scale=1.5]{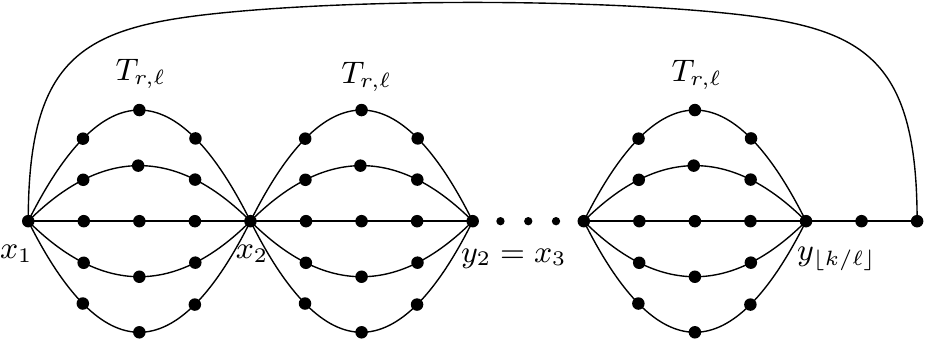}
	\end{center}
	\caption{Construction~\ref{undirected-construction}}
	\label{undir-figure}
\end{figure}

Now we turn our attention to the upper bound. First we get a bound on the number of edges in an $L$-cycle graph.

\begin{prop}
\label{prop:edges} If $G$ is an $L$-cycle graph on $n$ vertices, then the number of edges in $G$ is at most $|L|\cdot n$.
\end{prop}

\begin{proof}
By induction on $n$. If $n \le |L|$, then the statement is trivial. Let $G$ be an $L$-cycle graph on $n$ vertices. Let $P$ be a longest path in $G$ and let $v$ be an end-vertex. As all neighbors of $v$ lie on $P$, they should be at distance $\ell-1$ from $v$ on $P$ for some $\ell \in L$. Therefore $d_G(v) \le |L|$ and $e(G)\le e(G-v)+|L|\le |L|\cdot n$ by induction.
\end{proof}

The next lemma shows that in an $L$-cycle graph there cannot be too many short paths connecting two vertices. It was first proved, in a stronger form, by Lam and Verstra\"{e}te \cite{LV}. Our proof is different, we include it for the sake of completeness.

\begin{lemma}
\label{lem:paths} There exists a constant $c=c(L)$ such that if $G$ is an $L$-cycle graph and $2\ell$ is the smallest even number in $L$, then for any two vertices $x,y \in V(G)$ there are less than $c$ paths of length at most $\ell-1$ between them.
\end{lemma}

\begin{proof}
Let $f(i)$ be the maximum number of paths of length $i$ between two vertices $x$ and $y$ in $G$. First, $f(1)\le 1$ as the graph is simple. To prove the lemma, we show by induction that $f(j)-1\le (j-1)^2\cdot \max \{f(i)^2: 1\le i<j\}$ for $1<j\le l-1$. If there is at most one path of length $j$ between $x$ and $y$, then $f(j)-1\le 0$, we are done. Otherwise, any two paths of length $j$ between $x$ and $y$ intersect in a third vertex as otherwise there would be a cycle of length $2j<2\ell$. Let $P$ be one of the paths of length $j$ between $x$ and $y$. Then all the $f(j)-1$ other paths of length $j$ intersect it in at least one of its $j-1$ inner vertices. By the pigeonhole principle there is a third vertex $z$ on $P$ such that $(f(j)-1)/(j-1)$ of these paths of length $j$ include $z$. Again, by the pigeonhole principle, there is some index $i<j$ such that $z$ is the $(i+1)$st vertex for at least $(f(j)-1)/(j-1)^2$ many of these paths. Now split each of these paths into a path of length $i$ from $x$ to $z$ and a path of length $j-i$ from $z$ to $y$. Suppose there are $p$ many paths from $x$ to $z$ and $q$ many from $z$ to $y$. Then $pq$ is at least $(f(j)-1)/(j-1)^2$. Therefore, without loss of generality we may assume that $p$ is at least $\sqrt{f(j)-1}/(j-1)$. On the other hand, $p$ is clearly at most $f(i)$, so 
$f(j)-1\le (j-1)^2 f(i)^2$.

\end{proof}

Having \pref{edges} and \lref{paths} in hand, we can achieve a weak bound on the number of cycles of length $m \in L$ in the following way: Consider $\lceil {m}/{\ell}\rceil$ many edges $e_1,e_2,\dots,e_{\lceil {m}/{\ell}\rceil}$ in $G$. By \lref{paths}, there exist at most $c^{\lceil {m}/{\ell}\rceil}$ cycles $C$ in $G$ such that $e_j$ is the $[(j-1)\ell+1]$st edge of $C$ (for some consecutive ordering of the edges of $C$). Each cycle of length $m$ is counted $2m$ times, therefore the total number of cycles of length $m$ in $G$ is at most 
\begin{equation} \label{eq:1}
\frac{1}{2m}c^{\lceil {m}/{\ell}\rceil}\binom{|L|\cdot n}{\lceil {m}/{\ell}\rceil}=O_L(n^{\lceil {m}/{\ell}\rceil}) = O_L(n^{\lceil {k}/{\ell}\rceil})
\end{equation}

as $\lceil {m}/{\ell}\rceil\le \lceil {k}/{\ell}\rceil$ for any $m\in L$.
Note that if $\ell$ divides $k$, then the above proof yields the bound of \tref{ordermagn}, in all other cases it is off by a factor of $n$. 

In order to prove the correct upper bound we need an additional lemma.

\begin{lemma}\label{viszonylaguj} Let $c$ be the constant of \lref{paths}. Let $G$ be an $L$-cycle graph such that $2\ell$ is the smallest even integer in $L$. Suppose that there is a family $\cal P$ of at least $(cri)^{2i-2}$ paths in $G$ of length $i\ge 1$ between two different vertices $u$ and $v$. Then $G$ contains a path $P\in \cal P$ 
	which is the union of three paths $P_1,P_2,P_3$ where $P_1$ goes from $u$ to $u'$, $P_2$ goes from $u'$ to $v'$ and  $P_3$ goes from $v'$ to $v$  (we allow $P_1$ and $P_2$ to be empty, i.e., $u=u'$ and $v=v'$) such that there exists a Theta-graph $T_{r,t}$ (for some $t$ with $2\le t\le i$) with main vertices $u'$ and $v'$ and $V(P)\cap V(T_{r,t})=V(P_2)$.
	
\end{lemma}

\begin{proof} We proceed by induction on $i$. The statement of the lemma holds for any $1 \le i \le \ell-1$ as \lref{paths} shows that no such graph exists. Let $i$ be at least $\ell$. During the proof a path always means a path in $\cal P$ and a subpath means a subpath of a path in $\cal P$. If there are $r$ disjoint paths between $u$ and $v$ then we are done. So we can assume that there are at most $r-1$ disjoint paths of length $i$ from $u$ to $v$. Their union has at most $ri$ vertices and every other path of length $i$ from $u$ to $v$ intersects this vertex set, thus there is a vertex $w$ that is contained in at least $(cri)^{2i-3}$ of these paths. This $w$ can be in different positions in those paths, but there are at least $(cri)^{2i-4}$ paths where $w$ is the $(p+1)$st vertex with $1\le p<i$. Then there are either at least $(cri)^{2p-2}> (crp)^{2p-2}$ subpaths of length $p$ from $u$ to $w$ or at least $(cri)^{2(i-p)-2}> (cr(i-p))^{2(i-p)-2}$ subpaths of length $i-p$ from $w$ to $v$. By induction on the appropriate family of subpaths we can find the required subpath and $T_{r,t}$ between $u$ and $w$ or between $w$ and $v$ which extends to a path as required.
\end{proof}


We are ready to prove the upper bound of \tref{ordermagn}. We show that for every $m\in L$ there are $O_L(n^{\lfloor k/\ell \rfloor})$ cycles of length $m$ in $G$. If $\ell$ divides $k$ or $m \le \ell \lfloor k / \ell \rfloor$, then the bound in (\ref{eq:1}) implies that the number of cycles of length $m$ is $O_L(n^{\lceil m/\ell \rceil})=O_L(n^{\lfloor k/\ell \rfloor})$. Therefore, we may assume that $k\ge m>\ell\lfloor k/\ell \rfloor \ge 2\ell$ (as $k\ge 2\ell$). Let $\ell'=m-\ell {\lfloor k/\ell \rfloor}$. Note that $0<\ell'<\ell$.

For a cycle $C$ of length $m$ let $x_1, \dots, x_{m}$ be the vertices in the natural order, so that we can talk about a subpath from $x_i$ to $x_j$, denoted by $P(x_i,x_j)$ without confusion. 
Fix a path $P(x_i,x_j)$ on $C$. 
A Theta-graph $T$ is \emph{parallel} to $P(x_i,x_j)$ if $T$ has main vertices $x_i$ and $x_j$, its vertices are otherwise disjoint from $P(x_i,x_j)$ and its main paths are of the same length as $P(x_i,x_j)$.

We call a vertex $x_i$ \emph{T-rich} (with respect to $C$) if there exists a Theta-graph $T_{k,t}$ parallel to a subpath of $P(x_i, x_{i+\ell+\ell'-1})$ (where the index $i+\ell+\ell'-1$ is modulo $m$). Note that $t\ge \ell>\ell'$ as $T_{k,t}$ contains cycles of length $2t$ while $2\ell$ is the length of the shortest even cycle. A vertex is \emph{T-poor} if it is not T-rich.

\begin{claim}
\label{clm:bad}
Every cycle $C$ of length $m>\ell\lfloor k/\ell \rfloor$ contains a T-poor vertex. 
\end{claim}

\begin{proof}[Proof of Claim.] Suppose not, i.e., every vertex of a cycle $C$ of length $m$ is T-rich. Therefore, for each vertex $x_i$ of $C$, there exists a Theta-graph parallel to a subpath of $P(x_i, x_{i+\ell+\ell'-1})$. Let $T$ be such a Theta-graph $T_{k,t}$ where that $t$ is minimal. Without loss of generality, we may assume that $x_1$ and $x_{t+1}$ are the main vertices of $T$. As $x_2$ is T-rich, there is a copy $T'$ of $T_{k,t'}$ that is parallel to a subpath of $P(x_2, x_{\ell+\ell'+1})$. Let $x_j$ and $x_{t'+j}$ be the main vertices of $T'$. Thus  $1<j\le 2+\ell+\ell'-1-t'\le \ell'+1<t+1$ as $t'\ge t \ge \ell >\ell'$. Therefore, $x_j$ occurs before $x_{t+1}$ on $C$. Furthermore, by the minimality of $t$, the vertex  $x_{t'+j}$ occurs after $x_{t+1}$ on $C$. Also, by definition, $j+t'\le 2+\ell+\ell'-1\le 2\ell\le m$.

The main vertices of $T$ and $T'$ are disjoint. Let us fix a main path $P$ of $T$. It has less than $k$ vertices and therefore intersects fewer than $k$ of the main paths of $T'$. Therefore, there is a main path $P'$ of $T'$ that is disjoint from $P$.	

Going from $x_1$ to $x_{t+1}$ on $P$, then to $x_j$ on the cycle $C$, then to $x_{t'+j}$ on $P'$, and finally to $x_1$ on $C$ we create a cycle $C'$ (see Figure \ref{undirected-pf-fig}). The length of $C'$ is $t+(t+1-j)+t'+(m-(j+t'-1)) = m+2(t+1-j)$. Because $t \geq \ell$ and $j \leq \ell'+1$ we have that $C'$ is of length at least $m+2(\ell-\ell') > m+\ell-\ell'=m+\ell-(m-\ell\lfloor k/\ell\rfloor)\ge k$; a contradiction.
\end{proof}

\begin{figure}[h]
	\begin{center}
		\includegraphics[scale=1.3]{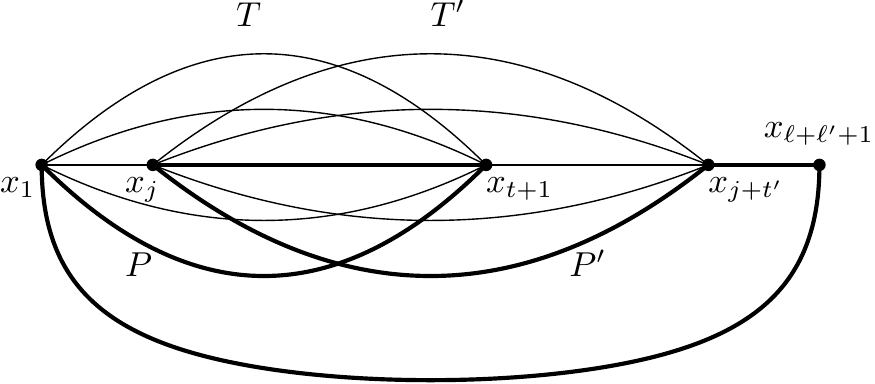}
	\end{center}
	\caption{Construction of the cycle $C'$.}
	\label{undirected-pf-fig}
\end{figure}

	
\begin{claim}
\label{clm:theta}
Let $C$ be a cycle of length $m>\ell\lfloor k/\ell \rfloor$ with a subpath $P(x_1,x_{m-\ell-\ell'+1})$  such that  $x_{m-\ell-\ell'+1}$ is T-poor (with respect to $C$). Then there exists a constant $c'_L$ such that there are at most $c'_L$ cycles of length $m$ containing $P(x_1,x_{m-\ell-\ell'+1})$ as a subpath and in which $x_{m-\ell-\ell'+1}$ is T-poor (with respect to this cycle).
\end{claim}

\begin{proof}[Proof of Claim] For large enough $c$ if there are more than $c$ such cycles then by Lemma \ref{viszonylaguj} for one of these cycles $C'$ there exists a Theta-graph $T_{k,t}$ $(t\ge 1)$ parallel to a subpath of the  path $C- P(x_1,x_{m-\ell-\ell'+1})$ of length $\ell+\ell'$ contradicting that $x_{m-\ell-\ell'+1}$ is a T-poor vertex in $C'$. Thus there are at most $c$ such cycles, as required.
\end{proof}

\noindent We are ready to count the cycles of length $m$ in $G$. Let $a={\lfloor k/\ell \rfloor}$. For every $a$-tuple $(e_1,e_2,\dots,e_a)$ of edges let $\cC_{e_1,e_2,\dots,e_a}$ be the set of those cycles $C$ for which $e_i$ is the $[(i-1)\ell+1]$st edge in $C$ (for some consecutive ordering of the edges of $C$) and the end vertex of $e_a$ that is further from $e_{a-1}$ is a T-poor vertex in $C$. By \clref{bad}, every cycle of $G$ contains a T-poor vertex and so it is contained (with at least one consecutive ordering) in at least one $\cC_{e_1,e_2,\dots,e_a}$. The number of such $a$-tuples is ${E(G) \choose a}a!=O_L(n^a) = O_L(n^{\lfloor k/\ell \rfloor})$.

For a fixed $a$-tuple $(e_1,e_2,\dots,e_a)$ we claim that $\cC_{e_1,e_2,\dots,e_a}$ contains a constant number of cycles $C$. Indeed, there are $2^a=O(1)$ ways to orient the edges $(e_1,e_2,\dots,e_a)$ in $C$. By \lref{paths}, for edges $e_i$ and $e_{i+1}$ when $i=1,2,\dots, a-1$, there is a constant number of paths of length $\ell-1$ between their endpoints. Once these are fixed, by \clref{theta}, there is only a constant number of cycles of length $m$ which contain these fixed edges as a subpath and in which the endvertex of this subpath incident to $e_a$ is T-poor.
\end{proof}

\section{Directed graphs}\label{directed-section}

In this section we prove \tref{dirasy} and \tref{direxact}. We begin with the construction that gives the lower bound in both theorems.

\begin{construction}\label{cons:directed} 
	Fix a single vertex in class $V_0$ and distribute the remaining $n-1$ vertices into $k-1$ classes $V_1,V_2,\dots, V_{k-1}$ of sizes as close as possible, i.e., of size $\lfloor\frac{n-1}{k-1}\rfloor$ or $\lceil \frac{n-1}{k-1}\rceil$. For each $i$ we add all possible arcs from $V_i$ to $V_{i+1}$ (where the index $i+1$ is modulo $k$).
\end{construction}

It is easy to see that every such digraph contains directed cycles only of length $k$ and the number of such cycles is the same for all such digraphs,  
\[
\prod_{i=0}^{k-2}\left\lfloor\frac{n-1+i}{k-1}\right\rfloor=(1+o(1))\left(\frac{n-1}{k-1}\right)^{k-1}.
\]

The next lemma shows that the order of magnitude of $\vec{c}(L,n)$ is given by \csref{directed}.

\begin{lemma}\label{dir1} For every fixed $L$ we have $\vec{c}(L,n)=\Theta(n^{k-1})$, where $k$ is the largest element of $L$. Moreover, there are $O(n^{k-2})$ cycles of length at most $k-1$.

\end{lemma}

\begin{proof} The lower bound is given by \csref{directed}. For the upper bound we prove $\vec{c}(L,n) \le |L|^2n^{k-1}$ by induction on $n$ and also that the number of cycles at most $k-1$ is at most $|L|^2n^{k-2}$. As in the undirected case, there is a vertex $v$ with outdegree at most $|L|$. (Note that it does not imply that the number of the arcs would be linear). Indeed, let $v$ be the endvertex of a longest directed path. Each outgoing arc from $v$ must return to the path, thus these arcs create cycles of different lengths. Therefore $v$ has outdegree at most $|L|$.

Let us examine the number of cycles of length $m \leq k$ containing $v$. There are at most $|L|$ choices for the vertex occurring after $v$ in the cycle and at most $n$ choices for the remaining $m-2$ vertices. Therefore, there are at most $|L|n^{m-2}$ cycles of length $m$ containing $v$ for a total of at most $|L|^2n^{k-2}$ cycles containing $v$ and for a total of at most $|L|^2n^{k-3}$ cycles containing $v$ and of length at most $k-1$.

By induction there are at most $|L|^2(n-1)^{k-1}$ cycles not containing $v$ and $|L|^2(n-1)^{k-2}$ cycles of length at most $k-1$ not containing $v$. Together with the cycles containing $v$ there are at most $|L|^2n^{k-1}$ cycles and at most $O(n^{k-2})$ cycles of length at most $k-1$.
\end{proof}

In order to prove Theorem~\ref{thm:dirasy} we need the following technical lemma.

\begin{lemma}\label{dir2} Fix $m \leq n$. Let $a_1,a_2,\dots, a_n$ be a sequence of non-negative integers such that $\sum_{i=1}^n a_i=n$. Then we have  
	\[
	\sum_{j=1}^{n-m+1} (a_j a_{j+1} \cdots a_{j+m-1}) \le (n/m)^m.
	\]

\end{lemma}

\begin{proof} 
Consider sequences satisfying the assumption of the lemma such that the left-hand side of the inequality is maximized. Among those sequences choose one such that the index $p$ of the last non-zero member $a_p$ is minimized.

If $p < m$, then every term of the sum is $0$ so the left-hand side is zero and we are done. If $p > m$,
then we construct a sequence $b_1,b_2,\dots, b_n$ as follows. Let $b_p = 0$ and $b_{p-m} = a_{p-m}+a_p$ and $b_i = a_i$ for $i \not = p, p-m$. Clearly $\sum_{i=1}^n b_i = n$.

Observe that for $j\leq p-m-1$ we have 
\[
a_j a_{j+1} \cdots a_{j+m-1} \leq b_j b_{j+1} \cdots b_{j+m-1}.
\]
Thus 
\[
\sum_{j=1}^{p-m-1} (a_j a_{j+1} \cdots a_{j+m-1}) \le \sum_{j=1}^{p-m-1} b_j b_{j+1} \cdots b_{j+m-1}.
\]
Furthermore, the sum of the next two terms is
\[
a_{p-m} \cdots a_{p-1} + a_{p-m+1}\cdots a_{p} = (a_{p-m}+a_p)a_{p-m+1}\cdots a_{p-1} = b_{p-m} \cdot b_{p-m+1} \cdots b_{p-1}.
\]
The remaining terms $a_j a_{j+1} \cdots a_{j+m-1}$ for $j \geq p-m+2$ are $0$.
Therefore
\[
\sum_{j=1}^{n-m+1} (a_j a_{j+1} \cdots a_{j+m-1}) \leq \sum_{j=1}^{n-m+1} (b_j b_{j+1} \cdots b_{j+m-1})
\]
which is a contradiction as the index of the last non-zero member among the $b_i$s is less than $p$.

Therefore, we are left with the case $p=m$. In this case the sum has a single term that is clearly at most $(n/m)^m$ which completes the proof of the lemma.
\end{proof}

\noindent We are now ready to prove our first main result concerning directed graphs.

\begin{proof}[Proof of \tref{dirasy}] 
	Let $\vec G$ be an $n$-vertex $L$-cycle digraph where $L$ has largest element $k$.
	By Lemma \ref{dir1} there are $O(n^{k-2})$ cycles of length less than $k$, so let us count the cycles of length exactly $k$.
	
	For a vertex $u$ let $p(u)$ be the length of a longest path ending at $u$. We call an arc $uv$ \emph{good} if $p(u)<p(v)$ and \emph{bad} otherwise. Note that every cycle contains at least one bad arc. Now we show that every vertex $u$ has at most $k$ bad outarcs. Let $P$ be a path of length $p(u)$ with endvertex $u$. If $uv$ is an arc and $v$ is not in $P$, then $P$ together with $uv$ is a path which implies $p(u)<p(v)$. In this case $uv$ is a good arc. If $v$ is in $P$, then we have a cycle. The cycles of this form are of different lengths and therefore there are at most $|L|<k$ of them.

Hence altogether the graph contains at most $kn$ bad arcs. Now we show that the number of cycles containing at least two bad arcs is $O(n^{k-2})$. Let us consider first the cycles containing two consecutive bad arcs. There are $kn$ ways to pick a bad arc and at most $k$ bad arcs leaving the endpoint, so there are at most $k^2n$ ways to pick two consecutive bad arcs. There are $O(n^{k-3})$ ways to pick the remaining $k-3$ vertices. Now let us consider cycles with two disjoint bad arcs. There are ${kn \choose 2}$ ways to choose two bad arcs, $k-1$ ways to choose in what distance these bad arcs are on the cycle and $O(n^{k-4})$ ways to pick the remaining $k-4$ vertices.

Now let us count cycles $C$ of length $k$ with exactly one bad arc $uv$. All other arcs $xy$ on $C$ are good, so $p(x) < p(y)$. Going along the path of length $k-1$ from $v$ to $u$ on $C$ we get that $p(u) \geq p(v) + k-1$.
On the other hand, every path $P$ of length $p(u)$ which ends in $u$ must contain $v$ (otherwise $p(u) < p(v)$ and $uv$ would be good). The distance between $v$ and $u$ on $P$ is at most $k-1$ as $P$ and $uv$ create a cycle. Thus the subpath of $P$ starting at the first vertex of $P$ and ending in $v$ is a path of length at least $p(u)-k+1$, which implies that $p(v)\ge p(u)-k+1$, that is, $p(u)\le p(v)+k-1$. Therefore, $p(u)=p(v)+k-1$.

Put $A_i=\{v \in V(\vec G): p(v)=i\}$. For each cycle $C$ of length $k$ with exactly one bad arc, there exists a $j$ such that $C$ has exactly one vertex in each class $A_j, A_{j+1},\dots, A_{j+k-1}$ (the bad arc $uv$ is between $A_{j+k-1}$ and $A_j$). We would like to show that if we attempt to build a cycle $C$ by choosing a vertex from each class $A_{j+1}, A_{j+2}\dots, A_{j+k-1}$, then there is at most one choice for the vertex in $A_j$. Suppose that $v \in A_j$ is such a vertex and $u$ is the vertex chosen from $A_{j+k-1}$. Let $P$ be a path of length $j+k-1$ ending in $u$. The vertex $v$ must be on $P$ otherwise there would be a path of length greater than $j+k-1>j$ ending in $v$; a contradiction (as $v \in A_j$). Now suppose that $v$ is the $i$th vertex on $P$. If $i-1>j$, then there is a path of length $i-1>j$ ending in $v$; a contradiction. If $i-1<j$, then as $u$ is the $(j+k)$th vertex of $P$, the subpath of $P$ from $v$ to $u$ and the arc $uv$ form a cycle of length $j+k-(i-1)>k$; a contradiction.

Therefore, if we build a cycle $C$ by choosing a single vertex from each $A_{j+1},\dots, A_{j+k-1}$, then we have at most one choice for the vertex in $A_j$.
Therefore, the number of cycles of length $k$ with exactly one bad arc is at most
\[
\sum_{j=1}^{n-k+1}\prod_{i=j+1}^{j+k-1} |A_i|.
\]
By Lemma~\ref{dir2} this is at most
\[
\left(\frac{n}{k-1}\right)^{k-1}.
\]
\end{proof}

We start our investigations of $\{k\}$-cycle directed graphs by the following structural lemma.

\begin{lemma}
\label{lem: partite} 
Let $\vec{G}$ be a strongly connected $\{k\}$-cycle directed graph. Then the vertex set of $\vec{G}$ can be partitioned into $k$ classes $V_0 \cup V_1 \cup \dots \cup V_{k-1}$, such that for any arc $uv$ in $\vec{G}$ there exists an $1\le i \le k$ such that $u\in V_i$ and $v\in V_{i+1}$ (where the index $i+1$ is modulo $k$).
\end{lemma}

\begin{proof}It is well known that a directed graph $\vec{G}$ is strongly connected if and only if it has an ear decomposition, i.e., $\vec{G}$ can be constructed by starting with a directed cycle and in each step adding a directed path (called an ear) between two (not necessarily distinct) vertices of the graph from the previous step.

We prove the statement by induction on the number of ears in an ear decomposition of $\vec{G}$.
If $\vec{G}$ is a directed cycle of length $k$ (i.e., there is no ear in the ear decomposition), then the vertex set can be partitioned as in the statement of the lemma.

Now suppose that the ear decomposition of $\vec{G}$ has at least one ear. Let $\vec{G'}$ be the digraph before adding the last ear $P$. By induction we may partition the vertices of $\vec{G'}$ as in the statement of the lemma. If the endpoints of $P$ are the same vertex in $\vec{G'}$, i.e., $P$ is a cycle, then it is of length $k$ and it is easy to see that the partition of $\vec{G'}$ can be extended to include all the vertices of $\vec{G}$.

Therefore, we may assume that $P$ is a directed path with vertices $\{p_1,\dots, p_s\}$ and that only $p_1$ and $p_s$ are vertices in $\vec{G'}$. The digraph $\vec{G'}$ is also strongly connected, so there is a directed path $Q$ in $\vec{G'}$ from $p_s$ to $p_1$. The directed paths $Q$ and $P$ together form a directed cycle which must be of length $k$. Thus $Q$ has length $k-s+1$. Let $V_0,V_1,\dots, V_k$ be the partition of $\vec{G'}$ as given by induction. Without loss of generality we may assume $p_s$ is in $V_0$ and therefore $p_1$ must be in class $V_{k-s+1}$. For $1<i<s$ let us add $p_i$ to class $V_{k-s+1+i}$ (indices are modulo $k$). This results in a vertex partition of $\vec{G}$ as desired.
\end{proof}

\begin{proof}[Proof of \tref{direxact}] 
	We begin by examining the case when $k=2$. Suppose that $\vec G$ is a $\{2\}$-cycle digraph on $n$ vertices. We can create an (undirected) graph $G$ by replacing each $2$-cycle of $\vec G$ with an edge. The graph $G$ is clearly a forest as a cycle in $G$ would imply the existence of a cycle of length longer than $2$ in $\vec G$. Therefore $\vec G$ has at most $n-1$ cycles.
		On the other hand, for any tree $G$ we can replace each edge of $G$ with a $2$-cycle to get a digraph $\vec G$ on $n$ vertices with $n-1$ cycles.
	
	Now we assume $k>2$ and proceed by induction on $n$. The case $k=n$ is trivial, so let $n>k$ and assume the statement holds for smaller graphs. 
Suppose $\vec{G}$ has strongly connected components of sizes $n_1,n_2,\dots, n_s$. Clearly, the vertices of a directed cycle are all in the same component. We apply induction to each component to conclude that the number of cycles in $\vec{G}$ is at most
\[
\sum_{i=1}^s\left(\frac{n_i-1}{k-1}\right)^{k-1} < \left(\frac{n-1}{k-1}\right)^{k-1}.
\]
We may now restrict our attention to the case when $\vec{G}$ is strongly connected.

Let $V_0,\dots, V_{k-1}$ be the partition of the vertices of $\vec{G}$ given by Lemma~\ref{lem: partite}. Without loss of generality, let us suppose that $V_0$ is a minimal size class in the partition. Let $\vec{G}^*$ be the digraph resulting from adding arcs to $\vec{G}$ such that all possible arcs from $V_i$ to $V_{i+1}$ for each $i$ (where the index $i+1$ is modulo $k$) are present. Clearly $\vec{G}$ is a sub(di)graph of $\vec{G}^*$ and therefore any cycle in $\vec{G}$ is also a cycle in $\vec{G}^*$. For an arbitrary cycle $C$ in $\vec{G}$, let $x,y,z$ be the vertices of $C$ in classes $V_0,V_{k-2},V_{k-1}$, respectively. Let $P$ be the path of length $k-2$ from $x$ to $y$ in $C$. The cycle $C$ is uniquely determined by $P$ and $z$; we say that $P$ and $z$ \emph{form} the cycle $C$.

For each class $V_i$, let us give an ordering of its vertices, so put $V_i = \{v^i_1,v^i_2,\dots, v^i_{|V_i|}\}$. Consider the collection of paths of length $k-2$ in $\vec{G}^*$ that start in $V_0$ and end in $V_{k-2}$. We say that two such paths $P=\{v^0_{p_0},v^1_{p_1},\dots, v^{k-2}_{p_{k-2}}\}$ and $Q=\{v^0_{q_0},v^1_{q_1},\dots, v^{k-2}_{q_{k-2}}\}$
are \emph{equivalent} if there exists a $d \le |V_0|-1$ such that $q_0-p_0=d$ and  $q_i-p_i\equiv d \;(\bmod\; |V_i|)$ for every $1\le i\le k-2$. It is easy to see that being equivalent is an equivalence relation on the set of paths (of length $k-2$ in $\vec{G}^*$ that start in $V_0$ and end in $V_{k-2}$). Clearly the collection of paths in one equivalence class are pairwise vertex disjoint. Furthermore, for every $v \in V_0$ each equivalence class contains exactly one path that begins with $v$. Therefore, each equivalence class is of size $|V_0|$.

Observe that the number of paths of length $k-2$ starting in $V_0$ and ending in $V_{k-2}$ in $\vec{G}^*$ is
\[
\prod_{i=0}^{k-2} |V_i|.
\]
Thus, the number of equivalence classes is 
\begin{equation}\label{num-classes}
\frac{1}{|V_0|}\prod_{i=0}^{k-2} |V_i| =\prod_{i=1}^{k-2} |V_i|.
\end{equation}

For each class of equivalent paths $\cal P$ we define an auxiliary (undirected) bipartite graph $H=H_{\cal P}$ with vertex classes $V_0$ and $V_{k-1}$ as follows.
For each $x \in V_0$ and $z \in V_{k-1}$ we add the edge $xz$ to $H$ if $\vec{G}$ has a path in $\mathcal{P}$ that begins in $x$ and together with $z\in V_{k-1}$ forms a cycle in $\vec{G}$.

\begin{figure}[h]
	\begin{center}
		\includegraphics[scale=.8]{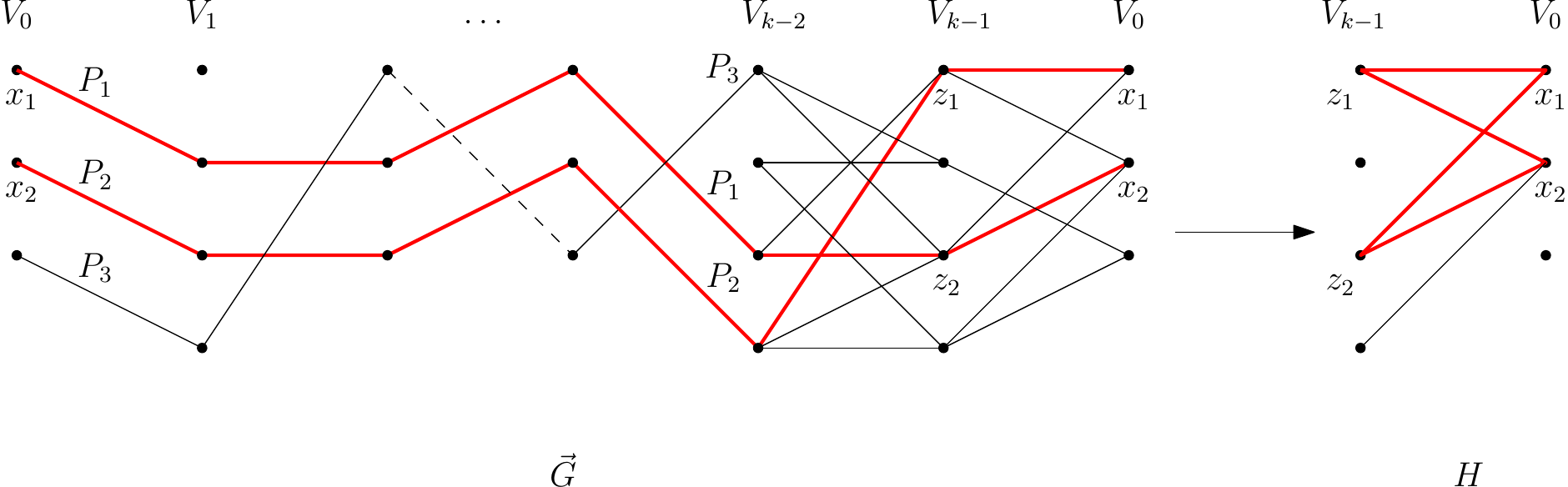}
	\end{center}
	\caption{The auxiliary graph $H$ corresponding to a class of equivalent paths $\cal P$ in $\vec G$.}
	\label{directed-pf-fig}
\end{figure}

We claim that the auxiliary graph $H$ is a forest. Suppose (to the contrary) that $H$ has a cycle with vertices $z_1,x_1,z_2,x_2,\dots ,z_s,x_s$ (in the natural order) where $x_i\in V_0$ and $z_i\in V_{k-1}$ (and $s>1$). Let $P_i$ be the (unique) path in $\cal P$ containing $x_i$. It is easy to see that by the definition of $H$ the sequence $z_1,P_1,z_2,P_2,\dots, z_s,P_s$ corresponds to a directed cycle of length $sk$ in $\vec{G}$; a contradiction. See Figure \ref{directed-pf-fig} for an illustration where the cycle in $H$ and its corresponding cycle in $\vec{G}$ are drawn with bold (and red) edges.
 Therefore, $H$ is a forest on $|V_0|+|V_{k-1}|$ vertices and therefore has at most $|V_0|+|V_{k-1}|-1$ edges. 

By definition, each cycle in
 $\vec{G}$ is associated with a unique edge in some auxiliary graph corresponding to a class of equivalent paths. 
 Therefore, by (\ref{num-classes}) we get that the total number of cycles in $\vec{G}$ is at most
 \[
 (|V_0|+|V_{k-1}|-1)\cdot \prod_{i=1}^{k-2}|V_i|.
 \]
 This is the product of $k$ positive integers whose sum is $n-1$. Thus this is maximal if all the integers are as close as possible in size. That is, all of $|V_0|+|V_{k-1}|-1,|V_1|,|V_2|,\dots, |V_{k-2}|$ are equal to $\lfloor \frac{n-1}{k-1}\rfloor$ or $\lceil \frac{n-1}{k-1}\rceil$. 
Therefore, the upper bound on the number of cycles in $\vec{G}$ is
\[
\prod_{i=0}^{k-2}\left\lfloor\frac{n-1+i}{k-1}\right\rfloor.
\]

It remains to prove that if $\vec{G}$ has 
$\prod_{i=0}^{k-2}\lfloor\frac{n-1+i}{k-1}\rfloor$ many cycles, then it is of the form given by Construction~\ref{cons:directed}. In order to have the maximal number of cycles, there must be equality in every bound throughout the proof of the upper bound. First, each of the $k-1$ terms $|V_0|+|V_{k-1}|-1,|V_1|,|V_2|,\dots, |V_{k-2}|$ must be equal to $\lfloor \frac{n-1}{k-1}\rfloor$ or $\lceil \frac{n-1}{k-1}\rceil$.
Second, each auxiliary graph $H$ defined above must have exactly $|V_0|+|V_{k-1}|-1$ many edges, i.e., $H$ is a spanning tree. Therefore for each auxiliary graph $H$, there are edges incident to every vertex in  $V_0$. By the definition of $H$, this implies that every path in the corresponding $\cal P$ forms at least one cycle in $\vec{G}$ (together with some vertex in $V_{k-1}$). As this is true for all equivalence classes, all arcs of $\vec{G}^*$ that are non-incident to $V_{k-1}$ must be present in $\vec{G}$ as well. 


We now show that $|V_0|=1$. Suppose (to the contrary) that $|V_0|>1$. We assumed that $V_0$ is a minimal size class, so we also have $|V_{k-1}|>1$. Consider an arbitrary auxiliary graph $H=H_{\cal P}$. It is a spanning tree so it must contain a path of length three $z_1,x_1,z_2,x_2$ such that $x_i\in V_{0}$ and $z_i\in V_{k-1}$. Let $P_i$ be the (unique) path in $\cal P$ that contains $x_i$ and let $P_i'$ be the path resulting from the removal of $x_i$ from $P_i$. We claim that  the sequence $z_1,x_1,P_2',z_2,x_2,P_1'$ corresponds to a cycle of length $2k$ in $\vec G$. Indeed, its arcs incident to $V_{k-1}$ exist as $y_1x_1y_2x_2$ is a path in $H$ while all its other arcs are non-incident to $V_{k-1}$ and therefore appear in $\vec{G}$ as proved above.

As $|V_0|=1$, we have that all arcs from $V_{k-1}$ to $V_0$ must be present in order for $\vec G$ to have as many cycles as given by the lower bound from Construction~\ref{cons:directed}.
\end{proof}

\section{Final remarks and open problems}
In this paper we addressed the problem of determining the maximum possible number of cycles in an $n$-vertex (directed) graph $G$ if cycle lengths in $G$ belong to $L$. These parameters are denoted by $c(L,n)$ (for undirected graphs) and $\vec{c}(L,n)$ (for directed graphs). In the undirected case our main result \tref{ordermagn} determined the order of magnitude of $c(L,n)$ for any fixed subset $L$ of the integers. Several natural questions arise:
determine the asymptotic behavior of $c(L,n)$ for a fixed set $L$. Does there exist an $L$-cycle graph that contains more cycles than the one in Construction~\ref{undirected-construction}? Two other problems are to find upper and lower bounds when the set $L=L_n$ has size that tends to infinity as $n$ grows; alternatively we may fix the size of $L$ and let the size of the elements tend to infinity. When the size of $L$ tends to infinity, the statement of Theorem~\ref{thm:ordermagn} does not hold as shown by the following example: begin with a cycle $C$ of length $n-s$ and $s$ additional vertices $y_1,y_2,\dots,y_s$. Let $x_1,x_2,\dots,x_s$ be $s$ consecutive vertices on the cycle $C$. For $1 \leq i \leq s$ connect $y_i$ to $x_i$ and $x_{i+1}$. It is easy to see that the graph has $s+2^s$ many cycles, but Theorem~\ref{thm:ordermagn} would give an upper-bound of $n^2$.

In the directed case, the asymptotic behavior of $\vec{c}(L,n)$ for any fixed set $L$ is given by \tref{dirasy}. Furthermore, \tref{direxact} shows that \csref{directed} is optimal if $L$ is a single integer $k$. However, it is easy to see that a lower order error term is needed for general sets $L$ of constant size. If $L=\{3,4,\dots,k\}$ for some $k\ge 3$, then to any directed graph $\vec{G}$ given by \csref{directed} we can add arcs $v_iv_j$ where $v_i \in V_i$ and $v_j \in V_j$ where $1\le i<j\le k-1$. This will create new directed cycles of length less than $k$, but no cycles of length more than $k$. Will adding all such arcs result in an optimal construction for $L=\{3,4,\dots,k\}$ or can we do better? As in undirected graphs, the case when $L$ (or its members) is not of constant size is also interesting.

\bibliographystyle{abbrv}
\bibliography{fewcycles}

\begin{thebibliography}{1}

\bibitem{AlSh}
N.~Alon and C.~Shikhelman.
\newblock Many {$T$} copies in {$H$}-free graphs.
\newblock {\em Journal of Combinatorial Theory, Series B}, 2016.

\bibitem{BoGy}
B.~Bollob{\'a}s and E.~Gy{\H{o}}ri.
\newblock Pentagons vs. triangles.
\newblock {\em Discrete Mathematics}, 308(19):4332--4336, 2008.

\bibitem{BoSi}
J.~A. Bondy and M.~Simonovits.
\newblock Cycles of even length in graphs.
\newblock {\em Journal of Combinatorial Theory, Series B}, 16(2):97--105, 1974.

\bibitem{D-book}
R.~Diestel.
\newblock {\em Graph theory}, volume 173 of {\em Graduate Texts in
  Mathematics}.
\newblock Springer, Heidelberg, fourth edition, 2010.

\bibitem{Er}
P.~Erd{\H o}s.
\newblock On some problems in graph theory, combinatorial analysis and
  combinatorial number theory.
\newblock {\em Graph Theory and Combinatorics (Cambridge, 1983), Academic
  Press, London}, pages 1--17, 1984.

\bibitem{Gr}
A.~Grzesik.
\newblock On the maximum number of five-cycles in a triangle-free graph.
\newblock {\em Journal of Combinatorial Theory, Series B}, 5(102):1061--1066,
  2012.

\bibitem{Gy}
E.~Gy{\H{o}}ri.
\newblock On the number of ${C}_5$'s in a triangle-free graph.
\newblock {\em Combinatorica}, 9(1):101--102, 1989.

\bibitem{HaETAL}
H.~Hatami, J.~Hladk{\`y}, D.~Kr{\'a}l', S.~Norine, and A.~Razborov.
\newblock On the number of pentagons in triangle-free graphs.
\newblock {\em Journal of Combinatorial Theory, Series A}, 120(3):722--732,
  2013.

\bibitem{LV}
T.~Lam and J.~Verstra{\"e}te.
\newblock A note on graphs without short even cycles.
\newblock {\em Journal of Combinatorics}, 12(1):N5, 2005.

\end{thebibliography}



\end{document}